\documentclass{amsart}
\usepackage{graphicx}
\usepackage{amsmath}%in Winedt
\usepackage{amssymb}
\usepackage{amsthm}
\usepackage{esint}

\newtheorem{theorem}{Theorem}[section]
\newtheorem{lem}[theorem]{Lemma}

\newtheorem{assumption}{Assumption}[section]

\numberwithin{equation}{section}

\begin{document}

\title[Prescribed mean curvature]{Prescribed mean curvature equation on the unit ball in the presence of reflection or rotation symmetry}

\author{Pak Tung Ho}
\address{Department of Mathematics, Sogang University, Seoul 121-742, Korea}
\email{paktungho@yahoo.com.hk, ptho@sogang.ac.kr}

\curraddr{Department of Mathematics, Princeton University, Fine Hall, Washington Road, Princeton NJ 08544 USA}

\email{ptho@princeton.edu}

%    General info

\date{2nd January, 2017.}


\begin{abstract}
Using the flow method, we prove some existence results
for the problem of prescribing the  mean curvature on the unit ball.
More precisely, we prove that there exists a conformal metric on the unit ball such that its
mean curvature is $f$, when $f$ possesses certain reflection or rotation symmetry.
\end{abstract}

\maketitle

\section{Introduction}

The problem of prescribing scalar curvature on a closed manifold has been studied extensively
for last few decades.
More precisely, let $(M,g)$ be an $n$-dimensional compact smooth Riemannian manifold without boundary. Given a smooth
function $f$ on $M$, can we find a metric conformal to $g$ such that its scalar curvature is $f$?
This has been studied in \cite{Aubin&Bismuth,Berger,Chen&Li,Ho1,Kazdan&Warner1,Kazdan&Warner2,Ouyang,Tang}.
When $(M,g)$ is the $n$-dimensional standard sphere $S^n$,
it is called the Nirenberg's problem and has been studied extensively. See \cite{Baird&Fardoun&Regbaoui,Chang&Gursky&Yang,Chang&Yang3,Chang&Yang1,Chang&Yang2,Chen&Li1,Han,Han&Li1,Schoen&Zhang,Wei&Xu,Xu&Yang}
and the references therein.
In particular, Chang-Yang
obtained in \cite{Chang&Yang3} a perturbation theorem which asserts
that there exists a conformal metric whose scalar curvature is equal to $f$,
provided that the degree condition holds for $f$ which is a positive Morse function
and is sufficiently closed to $n(n-1)$ in $C^0$ norm.
See also \cite{Chtioui&OuldAhmedou&Yacoub,Chtioui&OuldAhmedou&Yacoub1,Ho5,Ho2,Ho3,Ho4,Ngo&Zhang}
for related results of prescribing the Webster scalar curvature on
CR manifolds.

A geometric flow has been introduced to study the Nirenberg's problem
by Struwe in \cite{Struwe} for $n=2$, and has been generalized to $n\geq 3$
by Chen-Xu in \cite{Chen&Xu}.
More precisely, the scalar curvature flow
is defined as
$$\frac{\partial}{\partial t}u=\frac{n-2}{4}(\alpha f-R_{\tilde{g}})u,$$
where $R_{\tilde{g}}$ is the scalar curvature of $\tilde{g}=u^{\frac{4}{n-2}}g$ and $\alpha$ is a constant chosen to preserve the volume
along the flow. Using the scalar curvature flow, Chen-Xu \cite{Chen&Xu} was able to
prove Chang-Yang's result
with the quantitative bound on $\| f-n(n-1)\|_{C^0(S^n)}$.

Again using the scalar curvature flow, Leung-Zhou \cite{Leung&Zhou} proved an existence
result for prescribing scalar curvature with symmetry. To describe their result, we
have the following:
\begin{assumption}\label{assumption1.1}
$f$ is symmetric under a mirror reflection upon a hyperplane $\mathcal{H}\subset \mathbb{R}^{n+1}$
passing through the origin.
\end{assumption}

Under Assumption \ref{assumption1.1}, without loss of generality, we may assume that $\mathcal{H}$ is the hyperplane perpendicular to the
$x_1$-axis. Then the symmetry can be expressed as
$$f(\gamma_m(x))=f(x)$$
where $\gamma_m:S^n\to S^n$ is given by
$$\gamma_m(x_1,x_2,...,x_{n+1})=
(-x_1,x_2,...,x_{n+1})\hspace{2mm}\mbox{ for }x=(x_1,x_2,...,x_{n+1})\in S^n.$$
Then
$$\mathcal{F}=\mathcal{H}\cap S^n=\{(0,x_2,...,x_{n+1})\in S^n\}$$
is the fixed point set.

\begin{assumption}\label{assumption1.2}
$f$ is invariant under a rotation $\gamma_\theta$ of angle $\theta= \pi/ k$
with the rotation axis being a straight line in $\mathbb{R}^{n+1}$ passing through the origin.
Here $k>1$ is an integer.
\end{assumption}

Under Assumption \ref{assumption1.2}, without loss of generality, we may assume
that the rotation axis is the $x_{n+1}$-axis.
In this case,
$\mathcal{F}= \{N, S\}$ is the  fixed point set, where $N=(0,...,0,1)$ is the north pole
and $S=(0,...,0,-1)$ is the south pole.

With these assumptions, we can state the result of Leung-Zhou in \cite{Leung&Zhou}.
\begin{theorem}[Leung-Zhou \cite{Leung&Zhou}]
Suppose that $f$ is a positive smooth function on $S^n$ satisfying
Assumption $\ref{assumption1.1}$ or $\ref{assumption1.2}$.
Assume that
\begin{equation*}
x_m\in\mathcal{F}\mbox{ with }f(x_m)=\max_{\mathcal{F}}f\hspace{2mm}\Rightarrow\hspace{2mm}
\Delta_{S^n}f(x_m)>0
\end{equation*}
where $\Delta_{S^n}$ is the Laplacian of the standard metric of $S^n$, and
\begin{equation*}
\max_{S^n} f<2^{\frac{2}{n-2}}\Big(\max_{\mathcal{F}} f\Big).
\end{equation*}
Then $f$ can be realized as the scalar curvature of some metric conformal
to the standard metric of $S^n$.
\end{theorem}

Note that  existence results for prescribing scalar curvature with symmetry
were obtained earlier by Moser \cite{Moser} and by Escobar-Schoen \cite{Escobar&Schoen}.

The problem of prescribing the scalar curvature or the mean curvature has been studied on manifolds with boundary. See \cite{Chen&Ho,Chen&Ho&Sun,Djadli&Malchiodi&OuldAhmedou,Escobar1,Han&Li2,Sharaf,Zhang}
for example. In this paper, we consider the following problem of prescribing the mean curvature on the unit ball,
which is a natural analogy of the prescribing scalar curvature problem:
Let $B^{n+1}$ be the $(n+1)$-dimensional unit ball equipped with the flat metric $g_0$, i.e.
$$B^{n+1}=\Big\{(x_1,...,x_n,x_{n+1})\in\mathbb{R}^{n+1}: \sum_{i=1}^{n+1}x_i^2\leq 1\Big\}.$$
The boundary of $B^{n+1}$ is the $n$-dimensional unit sphere $S^n$, i.e.
$$\partial B^{n+1}=S^n=\Big\{(x_1,...,x_n,x_{n+1})\in\mathbb{R}^{n+1}: \sum_{i=1}^{n+1}x_i^2=1\Big\}.$$
The mean curvature $H_{g_0}$ of $S^n$ with respect to $g_0$
is equal to $1$, i.e.
$H_{g_0}=1$.
On the other hand, the metric induced by $g_0$ on $S^n$ is the standard metric $g_{S^n}$ of $S^n$.
We study the following problem: given a smooth function $f$ on $S^n$,
find a metric conformal to $g_0$ such that it is flat in the interior of $B^{n+1}$
and its mean curvature is equal to $f$ on $\partial B^{n+1}=S^n$.
The problem is equivalent to finding a positive harmonic function $u$ in the ball with nonlinear
boundary condition:
\begin{equation}\label{0}
\begin{split}
\Delta_{g_0} u&=0\hspace{2mm}\mbox{ in }B^{n+1},\\
\frac{2}{n-1}\frac{\partial u}{\partial\nu_{g_0}}+u&=fu^{\frac{n+1}{n-1}}
\hspace{2mm}\mbox{ on }S^{n}.
\end{split}
\end{equation}
Here, $\Delta_{g_0}$ is the Laplacian of $g_0$ and
$\displaystyle\frac{\partial}{\partial\nu_{g_0}}$ is the outward normal derivative of $g_0$.
See \cite{Abdelhedi&Chtioui,Abdelhedi&Chtioui&OuldAhmedou,Cao&Peng,Cherrier,Djadli&Malchiodi&OuldAhmedou,Escobar1,Escobar&Garcia,
Sharaf,Sharaf&Alharthy&Altharwi} and references therein for
the results related to this problem.
In particular, Chang-Xu-Yang proved in
\cite{Chang&Xu&Yang}
that (\ref{0}) has a solution when
$f$ is a positive Morse function satisfying
the degree condition
and being sufficiently closed to $1$ in $C^0$ norm.
By the method of geometric flow,
Xu-Zhang \cite{Xu&Zhang}
proved Chang-Xu-Yang's result
with the quantitative bound on $\| f-1\|_{C^0(S^n)}$.

Inspired by the result of Leung-Zhou \cite{Leung&Zhou}, we study the problem of
prescribing the mean curvature on the unit ball with symmetry.
We prove the following:

\begin{theorem}\label{main}
Suppose that $f$ is a positive smooth function on $S^n$ satisfying
Assumption $\ref{assumption1.1}$ or $\ref{assumption1.2}$.
Assume that
\begin{equation}\label{8}
x_m\in\mathcal{F}\mbox{ with }f(x_m)=\max_{\mathcal{F}}f\hspace{2mm}\Rightarrow\hspace{2mm}
\Delta_{S^n}f(x_m)>0
\end{equation}
where $\Delta_{S^n}$ is the Laplacian of the standard metric of $S^n$, and
\begin{equation}\label{9}
\max_{S^n} f<2^{\frac{1}{n-1}}\Big(\max_{\mathcal{F}} f\Big).
\end{equation}
Then $f$ can be realized as the mean curvature
of some conformal metric
on the unit ball $B^{n+1}$.
\end{theorem}

We remark that Escobar has also studied in \cite{Escobar1} the existence
of a positive solution to (\ref{0}) when $f$ has symmetry.
The proof of Theorem \ref{main} follows the arguments of Leung-Zhou in \cite{Leung&Zhou}.
We also remark that we have used the arguments of Leung-Zhou to study
the problem of prescribing Webster scalar curvature on the CR sphere with symmetry. See \cite{Ho2}.

\textit{Acknowledgements.} Part of the work was done while the author was visiting Princeton University
in 2016. He thanks Prof. Paul C. Yang for the invitation and is grateful to Princeton University for the kind hospitality.
The author was supported by the National Research Foundation of Korea (NRF) grant funded
by the Korea government (MEST) (No.201631023.01).

\section{The flow and its properties}

Let $B^{n+1}$ be the $(n+1)$-dimensional unit ball equipped with the flat metric $g_0$.
Then the $n$-dimensional sphere $S^n=\partial B^{n+1}$ is the boundary of $B^{n+1}$.
Let $f$ be a positive smooth function defined on $S^n$.
Given a smooth function $u_0$ in $B^{n+1}$, we consider the flow
\begin{equation}\label{1}
u_t=\frac{n-1}{4}(\alpha f-H_{g})u\hspace{2mm}\mbox{ on }S^n
\end{equation}
for $t\geq 0$
with the initial condition
$$u|_{t=0}=u_0,$$
where
\begin{equation}\label{2}
H_g=u^{-\frac{n+1}{n-1}}\left(\frac{2}{n-1}\frac{\partial u}{\partial\nu_{g_0}}+u\right)\hspace{2mm}\mbox{ on }S^n
\end{equation}
is the mean curvature of the conformal metric $g=u^{\frac{4}{n-1}}g_0$,
and
\begin{equation}\label{3}
\Delta_{g_0} u=0\hspace{2mm}\mbox{ in }B^{n+1}.
\end{equation}
Here, $\Delta_{g_0}$ is the Laplacian of $g_0$ and
$\displaystyle\frac{\partial}{\partial\nu_{g_0}}$ is the outward normal derivative of $g_0$.
Also, $\alpha=\alpha(t)$ is defined as
\begin{equation}\label{4}
\alpha=\left(\fint_{S^n} H_g u^{\frac{2n}{n-1}}d\mu\right)\Big/
\left(\fint_{S^n}f u^{\frac{2n}{n-1}}d\mu\right),
\end{equation}
where $d\mu$ is the volume form of the standard metric $g_{S^n}$ on $S^n$ and
$$\fint_{S^n}\varphi d\mu=\frac{1}{\omega_n}\int_{S^n}\varphi d\mu\mbox{ for all }\varphi\in C^\infty(S^n),$$
where $\omega_n=\displaystyle\int_{S^n}d\mu$ is the volume of $S^n$ with respect to $g_{S^n}$.

We define the functionals
\begin{equation}\label{5}
E[u]=\fint_{S^n}\left(\frac{2}{n-2}\frac{\partial u}{\partial\nu_{g_0}}+u\right)u d\mu
=\fint_{S^n}H_g u^{\frac{2n}{n-1}}d\mu
\end{equation}
and
\begin{equation}\label{6}
E_f[u]=E[u]\Big/\left(\fint_{S^n}f u^{\frac{2n}{n-1}}d\mu\right)^{\frac{n-1}{n}}.
\end{equation}
It follows from Lemma 2.2 in \cite{Xu&Zhang} that
$$\frac{d}{dt}E_f[u]\leq 0$$
along the flow (\ref{1}).
In particular,
we have
\begin{equation}\label{7}
E_f[u]\leq E_f[u_0]\hspace{2mm}\mbox{ for all }t\geq 0.
\end{equation}

It follows from section 2.3 in \cite{Xu&Zhang} that
(\ref{1})-(\ref{3}) has a unique solution
$u(x,t)$ on $B^{n+1}\times [0,\infty)$  such
that, given $T>0$, there exists a constant $C=C(T)$ such that
\begin{equation}\label{2.1}
C^{-1}\leq u(x,t)\leq C\hspace{2mm}\mbox{ for all }(x,t)\in B^{n+1}\times [0,T].
\end{equation}

\begin{lem}\label{lem2.1}
Given an isometry $\gamma: (S^n,g_{S^n})\to (S^n,g_{S^n})$, we assume that
\begin{equation}\label{2.2}
f(\gamma(x))=f(x)\hspace{2mm}\mbox{ and }\hspace{2mm}u_0(\gamma(x))=u_0(x)\mbox{ for all }x\in S^n.
\end{equation}
Let $u(x,t)$ be the solution of \eqref{1}-\eqref{3} with initial data $u_0$.
Then
\begin{equation}\label{2.3}
u(\gamma(x),t)=u(x,t)\hspace{2mm}\mbox{ for all }(x,t)\in S^n\times [0,\infty).
\end{equation}
\end{lem}
\begin{proof}
Since $\gamma: (S^n,g_{S^n})\to (S^n,g_{S^n})$ is isometry,
for any $\varphi\in C^\infty(B^{2n+1})$, we have
\begin{equation}\label{2.4}
\left(\frac{\partial(\varphi\circ\gamma)}{\partial\nu_{g_0}}\right)(x)=\left(\frac{\partial\varphi}{\partial\nu_{g_0}}\right)(\gamma(x))
\mbox{ for all }x\in S^n.
\end{equation}
Since
the mean curvature of $g(x,t)=u(x,t)^{\frac{4}{n-1}}g_0$ satisfies
\begin{equation*}
H_{g(x,t)}(x)=u(x,t)^{-\frac{n+1}{n-1}}\left(\frac{2}{n-1}\frac{\partial}{\partial\nu_{g_0}}u(x,t)+u(x,t)\right)
\mbox{ for }x\in S^n,
\end{equation*}
it follows from (\ref{2.3}) and (\ref{2.4}) that
\begin{equation}\label{2.5}
H_{g(x,t)}(\gamma(x))=H_{g(\gamma(x),t)}(x)\mbox{ for all }x\in S^n.
\end{equation}

By (\ref{1})-(\ref{4}), the flow can be written as
\begin{equation}\label{2.13}
\begin{split}
\Delta_{g_0} u&=0\hspace{2mm}\mbox{ in }B^{n+1},\\
u_t&=-\frac{n-1}{4}
u^{-\frac{2}{n-1}}\left(\frac{2}{n-1}\frac{\partial u}{\partial\nu_{g_0}}+u\right)
+\frac{n-1}{4}
\left(\frac{\int_{S^n} H_g u^{\frac{2n}{n-1}}d\mu}{\int_{S^n}f u^{\frac{2n}{n-1}}d\mu}\right)
f u\hspace{2mm}\mbox{ on }S^n.
\end{split}
\end{equation}
From (\ref{2.2}), (\ref{2.4}) and (\ref{2.5}),
we know that $u(\gamma(x),t)$ is also a solution of (\ref{2.13}) with initial value $u_0$.
Now the assertion (\ref{2.3}) follows from
the uniqueness result stated in
Lemma 4.8 in \cite{Chen&Ho}.
\end{proof}

For $r > 0$ and $x_0\in S^n$, set $B^+_r (x_0) = \{x \in B^{n+1} : d_{g_0}(x, x_0) < r \}$ and
$\partial' B^+_r (x_0) = \partial B^+_r (x_0) \cap S^n$.
The following lemma was proved in \cite{Xu&Zhang}: (see section 4.2 in \cite{Xu&Zhang})

\begin{lem}\label{lem2.2}
Let $u(t)$ be the solution of the flow \eqref{1}. Let $(t_k)_{k=1}^\infty$
be any time sequence with $t_k\to\infty$ as $k\to\infty$. Consider the sequence $u_k:=u(t_k)$
and corresponding metrics $g_k = u(t_k)^{\frac{4}{n-1}}g_0$. Then, up to a subsequence,
either\\
(i) the sequence $u_k$ is uniformly bounded in $W^{1,p}(S^n, g_{S^n})\hookrightarrow L^\infty(S^n)$ for some
$p > n$; or\\
(ii) there exists a subsequence of $u_k$ and finitely many points $x_1, . . . , x_L\in S^n$ such
that for any $r > 0$ and any $l \in \{1, . . . , L\}$ there holds
\begin{equation}\label{2.6}
\liminf_{k\to\infty}\left(\omega_n^{-1}\int_{\partial'B_r(x_l)}|H_{g_k}|^nd\mu_{g_k}\right)^{\frac{1}{n}}\geq 1.
\end{equation}
In addition, if the alternative (ii) occurs, the sequence $u_k$ is
also uniformly bounded in $L^p$ on any compact subset of $(S^n\setminus\{x_l,...,x_L \}, g_{S^n})$.
\end{lem}

\begin{lem}\label{lem2.3}
Suppose that Lemma \ref{lem2.2}(ii) occurs.
For a point $x_0\in S^n$ with
\begin{equation}\label{2.7}
\max_{S^n}f<2^{\frac{1}{n-1}}f(x_0),
\end{equation}
suppose that the initial data $u_0$ satisfies
\begin{equation}\label{2.8}
E_f[u_0]\leq\frac{1}{f(x_0)^{\frac{n-1}{n}}}+\epsilon.
\end{equation}
Then for $\epsilon>0$ small enough, we have $L=1$.
\end{lem}
\begin{proof}
Let $x_1, . . . , x_L\in S^n$ be the blow-up points. If $L>2$, let
\begin{equation}\label{2.9}
r<\min_{1\leq i<j\leq L}d_{g_0}(x_i,x_j).
\end{equation}
For any $\epsilon>0$, there exists $k$ being sufficiently large such that
\begin{equation}\label{2.10}
\begin{split}
L-\epsilon&\leq \sum_{l=1}^L\left(\omega_n^{-1}\int_{\partial'B_r(x_l)}|H_{g_k}|^nd\mu_{g_k}\right)^{\frac{1}{n}}\\
&\leq L^{1-\frac{1}{n}}\left(\omega_n^{-1}\sum_{l=1}^L\int_{\partial'B_r(x_l)}|H_{g_k}|^nd\mu_{g_k}\right)^{\frac{1}{n}}\\
&\leq L^{1-\frac{1}{n}}\left(\fint_{S^n}|H_{g_k}|^nd\mu_{g_k}\right)^{\frac{1}{n}}\\
&\leq L^{1-\frac{1}{n}}\alpha(t_k)\left(\fint_{S^n} f^nd\mu_{g_k}\right)^{\frac{1}{n}}+L^{1-\frac{1}{n}}\left(\fint_{S^n}|\alpha(t_k) f-H_{g_k}|^nd\mu_{g_k}\right)^{\frac{1}{n}},
\end{split}
\end{equation}
where the first inequality follows from H\"{o}lder's inequality,
the second inequality follows from (\ref{2.6}), and the third inequality follows from (\ref{2.9}).
By Lemma 3.2 in \cite{Xu&Zhang},
for any $p\geq 2$,
we have
\begin{equation*}
\int_{S^n}|\alpha(t)f-H_g|^p d\mu_{g}\to 0\hspace{2mm}\mbox{ as }t\to\infty.
\end{equation*}
In particular,
\begin{equation}\label{2.11}
\int_{S^n}|\alpha(t_k)-H_{g_k}|^n d\mu_{g_k}\to 0\hspace{2mm}\mbox{ as }k\to\infty.
\end{equation}
On the other hand, by (\ref{4})-(\ref{6}), we have
\begin{equation}\label{2.12}
\alpha=\left(\fint_{S^n} H_g u^{\frac{2n}{n-1}}d\mu\right)\Big/
\left(\fint_{S^n}f u^{\frac{2n}{n-1}}d\mu\right)=E_f[u]\Big/
\left(\fint_{S^n}f u^{\frac{2n}{n-1}}d\mu\right)^{\frac{1}{n}}.
\end{equation}
Combining (\ref{2.10})-(\ref{2.12}), we obtain
\begin{equation*}
\begin{split}
L-\epsilon&\leq L^{1-\frac{1}{n}}
E_f[u_k]\left(\fint_{S^n} f^nd\mu_{g_k}\right)^{\frac{1}{n}}\Big/
\left(\fint_{S^n}f d\mu_{g_k}\right)^{\frac{1}{n}}+o(1)\\
&\leq
L^{1-\frac{1}{n}}
E_f[u_k]\Big(\max_{S^n} f\Big)^{\frac{n-1}{n}}
+o(1)\\
&\leq
L^{1-\frac{1}{n}}
\left[\left(\frac{\max_{S^n} f}{f(x_0)}\right)^{\frac{n-1}{n}}+O(\epsilon)\right]
+o(1)\\
&\leq
L^{1-\frac{1}{n}}
\left(2^{\frac{1}{n}}+O(\epsilon)\right)
+o(1),
\end{split}
\end{equation*}
where we have used (\ref{2.8}) in the second last inequality,
and (\ref{2.7}) in the last inequality.
This implies that $L=1$ when $\epsilon>0$ is small enough.
\end{proof}

We have the following lemma regarding the blow-up point. Its proof can be found in \cite{Xu&Zhang}.

\begin{lem}\label{lem2.4}
Under the assumptions of Lemma \ref{lem2.2}(ii) and \ref{lem2.3},
there is a point $Q\in S^n$ such that the following statements hold:\\
(a) As $k\to\infty$, the metrics $g_k$ concentrate at $Q$ in the sense
described by (ii) in Lemma 4.2 in \cite{Xu&Zhang}. As a consequence, for any positive
number $\rho$,
$\displaystyle\max_{\partial'B_{\rho}(Q)}u_k$
cannot be uniformly bounded from above for all $k\gg 1$,\\
(b) $Q$ is a critical point of $f$,\\
(c) $\Delta_{S^n} f(Q)\leq 0$, and \\
(d) $\displaystyle\lim_{k\to\infty}E_f[u_k]=\frac{1}{f(Q)^{\frac{n-1}{n}}}$.
\end{lem}

\section{Proof of Theorem \ref{main}}

As before, the south pole is denoted by
$S=(0,...,0,-1)\in S^n$. Let
$$\mathbb{R}^{n+1}_+=\{(z_1,...,z_n,z_{n+1})\in\mathbb{R}^{n+1}: z_{n+1}\geq 0\}$$
be the  $(n+1)$-dimensional upper half Euclidean space equipped with the Euclidean metric $g_E$.
Consider the following map $\Sigma: B^{n+1}\setminus\{S\}\to\mathbb{R}^{n+1}_+$ given by
\begin{equation}\label{3.1}
\Sigma(x_1,...,x_n,x_{n+1})=\left(\frac{4\overline{x}}{|\overline{x}|^2+(1+x_{n+1})^2}, \frac{2(1-|\overline{x}|^2-x_{n+1}^2))}{|\overline{x}|^2+(1+x_{n+1})^2}\right)
\end{equation}
where $\overline{x}=(x_1,...,x_n)$.
Note that
\begin{equation}\label{3.2}
\Sigma(x)=\left(\frac{2\overline{x}}{1+x_{n+1}},0\right)\mbox{ for }x\in S^n\setminus\{S\}.
\end{equation}
Also, $\Sigma^{-1}:\mathbb{R}^{n+1}_+\to  B^{n+1}\setminus\{S\}$ is given by
\begin{equation}\label{3.3}
\Sigma^{-1}(\overline{z},z_{n+1})=\left(\frac{4\overline{z}}{|\overline{z}|^2+(2+z_{n+1})^2}, \frac{4-z_{n+1}^2-|\overline{z}|^2}{|\overline{z}|^2+(2+z_{n+1})^2}\right)
\end{equation}
where $(\overline{z},z_{n+1})\in \mathbb{R}^{n+1}_+$ and $\overline{z}=(z_1,...,z_n)$.
Note that $\Sigma^{-1}: (\mathbb{R}^{n+1}_+,g_E) \to(B^{n+1}\setminus\{S\},g_0)$
is a conformal map such that
\begin{equation}\label{3.4}
\left[\left(\frac{4}{(2+z_{n+1})^2+|\overline{z}|^2}\right)^{\frac{n-1}{2}}\right]^\frac{4}{n-1}g_E=(\Sigma^{-1})^*(g_0).
\end{equation}
For $u\in C^\infty(B^{n+1})$, let
$$v(z)=\left(\frac{4}{(2+z_{n+1})^2+|\overline{z}|^2}\right)^{\frac{n-1}{2}}u\big(\Sigma^{-1}(z)\big).$$
It follows from (\ref{3.4}) that
\begin{equation*}
\left[\left(\frac{4}{(2+z_{n+1})^2+|\overline{z}|^2}\right)^{\frac{n-1}{2}}\right]^\frac{n+1}{n-1}\left(\frac{2}{n-1}\frac{\partial}{\partial\nu_{g_0}}u+u\right)\big(\Sigma^{-1}(z)\big)
=\frac{2}{n-1}\frac{\partial}{\partial\nu_{g_E}}v(z).
\end{equation*}
This implies that the mean curvatures of $g=u^{\frac{4}{n-1}}g_0$ and $\tilde{g}=v^{\frac{4}{n-1}}g_E$
are related by
\begin{equation}\label{3.5}
u^{-\frac{n+1}{n-1}}\left(\frac{2}{n-1}\frac{\partial u}{\partial\nu_{g_0}}+u\right)=H_g(\Sigma^{-1}(z))=H_{\tilde{g}}(z)=
\frac{2}{n-1}v^{-\frac{n+1}{n-1}}\frac{\partial}{\partial\nu_{g_E}}v,
\end{equation}
Therefore, by (\ref{5}), (\ref{6}) and (\ref{3.5}), we have
\begin{equation}\label{3.6}
E_f[u]=\left(\omega_n^{-1}\int_{\partial\mathbb{R}_+^{n+1}}H_{\tilde{g}}(z)v(z)^{\frac{2n}{n-1}} dz\right)
\Big/\left(\omega_n^{-1}\int_{\partial\mathbb{R}_+^{n+1}} f(\Sigma^{-1}(z))v(z)^{\frac{2n}{n-1}}dz\right)^{\frac{n-1}{n}}.
\end{equation}

\begin{lem}\label{lem3.1}
For any point $x_0\in S^n$ and any positive number $\epsilon>0$,
there exists a function $u_0\in C^\infty(S^n)$ such that
\begin{equation*}
E_f[u_0]\leq\frac{1}{f(x_0)^{\frac{n-1}{n}}}+\epsilon.
\end{equation*}
Moreover, we can choose $u_0$ to be invariant under the reflection upon a hyperplane
passing through $x_0$ and the origin $0\in R^{n+1}$, and invariant under rotations with
axis passing through $x_0$ and $0$.
\end{lem}
\begin{proof}
As the situation is unchanged after a rotation of $S^n$, we may assume
that $x_0=N=(0,...,0,1)$ the north pole. For $\lambda>0$, we let
$$v_0(z)=\left(\frac{2\lambda}{|\overline{z}|^2+(z_{n+1}+\lambda)^2}\right)^{\frac{n-1}{2}}\mbox{
for } z=\Sigma(x).$$
Then $v_0$ satisfies the equation:
\begin{equation}\label{3.7}
\begin{split}
\Delta v_0&=0\mbox{ in }\mathbb{R}_+^{n+1},\\
\frac{2}{n-1}\frac{\partial v_0}{\partial z_{n+1}}+v_0^{\frac{n+1}{n-1}}&=0\mbox{ in }\partial\mathbb{R}_+^{n+1},
\end{split}
\end{equation}
for some $c_0>0$.
We choose $\lambda$ such that
$$\int_{\partial\mathbb{R}_+^{n+1}} v_0(z)^{\frac{2n}{n-1}}dz=\omega_n.$$
If $\tilde{g}=v_0(z)^{\frac{4}{n-1}}g_E$, then it follows from  (\ref{3.7}) that
\begin{equation}\label{3.8}
\int_{\partial\mathbb{R}_+^{n+1}}H_{\tilde{g}}(z)v_0(z)^{\frac{2n}{n-1}} dz=\omega_n.
\end{equation}
We estimate
\begin{equation}\label{3.9}
\begin{split}
&\int_{\partial\mathbb{R}_+^{n+1}} v_0(z)^{\frac{2n}{n-1}}f(\Sigma^{-1}(z))dz\\
&=\int_{\partial\mathbb{R}_+^{n+1}} v_0(z)^{\frac{2n}{n-1}}[f(\Sigma^{-1}(z))-f(\Sigma^{-1}(0))]dz
+f(\Sigma^{-1}(0))\int_{\partial\mathbb{R}_+^{n+1}} v_0(z)^{\frac{2n}{n-1}}dz\\
&=\int_{\partial\mathbb{R}_+^{n+1}} v_0(z)^{\frac{2n}{n-1}}[f(\Sigma^{-1}(z))-f(\Sigma^{-1}(0))]dz
+f(\Sigma^{-1}(0))\int_{\partial\mathbb{R}_+^{n+1}} v_0(z)^{\frac{2n}{n-1}}dz\\
&=\int_{B_\delta(0)} v_0(z)^{\frac{2n}{n-1}}[f(\Sigma^{-1}(z))-f(\Sigma^{-1}(0))]dz\\
&\hspace{4mm}+\int_{\partial\mathbb{R}_+^{n+1}\backslash B_\delta(0)} v_0(z)^{\frac{2n}{n-1}}[f(\Sigma^{-1}(z))-f(\Sigma^{-1}(0))]dz
+f(\Sigma^{-1}(0))\omega_n\\
&=O(\epsilon_1)+O\left(\left[\frac{\lambda}{\delta_1}\right]^{n}\right)+ f(N)\omega_n,
\end{split}
\end{equation}
since
$$|f(\Sigma^{-1}(z))-f(\Sigma^{-1}(0))|\leq \epsilon_1\mbox{ whenever }z\in B_{\delta_1}(0)$$
and
$$|v_0(z)|\leq C\left(\frac{\lambda}{\delta_1}\right)^{n}\mbox{ whenever }z\in \partial\mathbb{R}_+^{n+1}\setminus B_{\delta_1}(0)$$
for some uniform constant $C$.
By first choosing $\delta_1>0$ to be small enough so that $\epsilon_1>0$ is small and then choosing a small
$\lambda>0$ so that $\lambda\cdot\delta_1^{-1}$
to be small, we obtain from (\ref{3.6}), (\ref{3.8}) and (\ref{3.9}) that
\begin{equation*}
E_f[u_0]\leq\frac{1}{f(x_0)^{\frac{n-1}{n}}}+\epsilon.
\end{equation*}
Here,
\begin{equation}\label{3.10}
u_0(x)=\left(\frac{1}{1+|\overline{z}|^2}\right)^{\frac{n-1}{2}}v_0(z)=\left(\frac{1}{1+|\overline{z}|^2}\right)^{\frac{n-1}{2}}
\left(\frac{2\lambda}{|\overline{z}|^2+\lambda^2}\right)^{\frac{n-1}{2}}
\mbox{ for }z=\Sigma(x),
\end{equation}
where $x\in S^n$.
By (\ref{3.1})-(\ref{3.3}), we have
$$|\overline{z}|^2=\frac{2|\overline{x}|^2}{|1+x_{n+1}|^2}
=\frac{2(1-|x_{n+1}|^2)}{|1+x_{n+1}|^2}\mbox{ for }z=\Sigma(x),
$$
where $x\in S^n$.
That is, $u_0$ defined in (\ref{3.10}) depends only on $x_{n+1}$. One can verify the claimed symmetries
directly.
\end{proof}

We are ready to prove Theorem \ref{main}.

\begin{proof}[Proof of Theorem \ref{main}]
Without loss of generality, we may assume that
$$
S\in \mathcal{F}\hspace{2mm}\mbox{ and }\hspace{2mm}f(S)=\max_{\mathcal{F}}f$$
where $S=(0,...,0,-1)\in S^n$ is the south pole.
We claim that there is $\delta>0$ such that
\begin{equation}\label{4.1}
f(S)\geq f(x_c)+\delta\mbox{ for any point }x_c\in \mathcal{F}\mbox{ with }
\Delta_{S^n} f(x_c)\leq 0.
\end{equation}
If not, then there exists a sequence of points $x_{c_i}\in\mathcal{F}$ such that
\begin{equation*}
\lim_{i\to\infty}f(x_{c_i})=f(S)=\max_{\mathcal{F}}f\hspace{2mm}\mbox{ and }\hspace{2mm}\Delta_{S^n} f(x_{c_i})>0.
\end{equation*}
By passing to subsequence, we assume that $x_{c_i}\to x_m\in\mathcal{F}$ as $i\to\infty$
such that
\begin{equation*}
f(x_{m})=\max_{\mathcal{F}}f\hspace{2mm}\mbox{ and }\hspace{2mm}\Delta_{S^n} f(x_{m})\geq 0,
\end{equation*}
which contradicts (\ref{8}). This proves (\ref{4.1}).

It follows from (\ref{4.1}) that
\begin{equation}\label{4.2}
\frac{1}{f(S)^{\frac{n-1}{n}}}+\epsilon<\frac{1}{f(x_c)^{\frac{n-1}{n}}}\hspace{2mm}\mbox{ for }x_c\in\mathcal{F} \mbox{ with }\Delta_{S^n} f(x_c)\leq 0,
\end{equation}
where $\epsilon>0$ is a small positive number.
Let $u_0$ be the positive smooth function constructed
in Lemma \ref{lem3.1}.
We claim that, with this choice of initial data, Lemma \ref{lem2.1}(i) occurs.
Suppose not, Lemma \ref{lem2.1}(ii) occurs.
It follows from Lemma \ref{lem2.3} that $L=1$.
Let $Q$ be the blow-up point.

We are going to show that
$Q\in\mathcal{F}$. Suppose  $Q\not\in \mathcal{F}$.
Then there exists an isometry $\gamma$ described in Assumption \ref{assumption1.1}
or \ref{assumption1.2} such that
\begin{equation*}
\gamma(Q)\neq Q,
\end{equation*}
which implies that
\begin{equation}\label{4.3}
\max_{\partial'B_r(\gamma(Q))}u_k\leq C\mbox{ for all }k\gg 0
\end{equation}
whenever $r$ is small enough,
since $\{u_k\}$ is uniformly bounded on any compact subsets of $S^n\setminus\{Q\}$
by Lemma  \ref{lem2.2}.
But Lemma \ref{lem2.1} implies that
$$u_k(\gamma(x))=u(\gamma(x),t_k)=u(x,t_k)=u_k(x)\mbox{ for all }x\in S^n.$$
This together with (\ref{4.3}) implies that
$\displaystyle\max_{\partial'B_r(Q)}u_k$ is uniformly bounded,
which contradicts Lemma \ref{lem2.4}(a).
This proves that $Q\in\mathcal{F}$.

Hence, by Lemma \ref{lem2.4}(c),
$Q\in\mathcal{F}$ and $\Delta_{S^n} f(Q)\leq 0$.
This together with (\ref{4.2})
implies that
$$\frac{1}{f(S)^{\frac{n-1}{n}}}+\epsilon<\frac{1}{f(Q)^{\frac{n-1}{n}}}.$$
Combining this with Lemma \ref{lem2.4}(d) and Lemma \ref{lem3.1}, we obtain
$$E_f[u_0]\leq\frac{1}{f(S)^{\frac{n-1}{n}}}+\epsilon<\frac{1}{f(Q)^{\frac{n-1}{n}}}=\lim_{k\to\infty}E_f[u_k],$$
which contradicts (\ref{7}).

Therefore, Lemma \ref{lem2.1}(i) occurs.
By Lemma 4.2(i) in \cite{Xu&Zhang}, $u_k$ converges to $u_\infty$ along the flow (\ref{1})-(\ref{3})
such that
\begin{equation*}
\begin{split}
\Delta_{g_0}u_\infty&=0\hspace{2mm}\mbox{ in }B^{n+1},\\
\frac{2}{n-1}\frac{\partial u_\infty}{\partial\nu_{g_0}}+u_\infty&=\alpha_\infty f u_\infty^{\frac{n+1}{n-1}}\hspace{2mm}\mbox{ on }S^n
\end{split}
\end{equation*}
for some $\alpha_\infty>0$.
This proves Theorem \ref{main}.
\end{proof}

\bibliographystyle{amsplain}

\end{document}